\documentclass[11pt]{amsart}
\pdfoutput=1
\usepackage{pgfplots}
\pgfplotsset{compat=1.18}
\usepgfplotslibrary{fillbetween}

\usepackage{subcaption}
\usepackage{amsmath, amsthm, amssymb, enumerate, hyperref, xcolor, tikz-cd, listings,hyperref, MnSymbol, textcomp}
\usepackage{tikz}
\usepackage{float}

\usepackage{fullpage}

\newcommand{\makeheading}[1]%
        {\hspace*{-\marginparsep minus \marginparwidth}%
         \begin{minipage}[t]{\textwidth}%
                {\large \bfseries #1}\\[-0.15\baselineskip]%
                 \rule{\columnwidth}{1pt}%
         \end{minipage}}

 \usepackage{mathtools}

\usepackage{amsmath, amsthm, amssymb, enumerate, hyperref, xcolor, tikz-cd, listings,hyperref,comment}

\usepackage[tt=false, type1=true]{libertine}
\usepackage[varqu]{zi4}
\usepackage[libertine]{newtxmath}
\usepackage[T1]{fontenc}
\usepackage{parskip}

\renewcommand{\i}{\iota}

\newcommand{\K}{\mathbb{K}}

\newcommand{\R}{\mathbb{R}}

\newcommand{\J}{\mathcal{J}}

\renewcommand{\hat}{\widehat}

\renewcommand\i{\mathrm{i}}

\theoremstyle{plain}

\newtheorem{prop}{Proposition}
\newtheorem{lemma}[prop]{Lemma}
\newtheorem{theorem}{Theorem}

\newtheorem{corollary}[theorem]{Corollary}

\newtheorem{teo}[prop]{Theorem}

\theoremstyle{definition}
\newtheorem{rem}[prop]{Remark}

\renewcommand{\phi}{\varphi}

\renewcommand{\epsilon}{\varepsilon}

\def\semicolon{;}
\def\applytolist#1{
    \expandafter\def\csname multi#1\endcsname##1{
        \def\multiack{##1}\ifx\multiack\semicolon
            \def\next{\relax}
        \else
            \csname #1\endcsname{##1}
            \def\next{\csname multi#1\endcsname}
        \fi
        \next}
    \csname multi#1\endcsname}

\def\calc#1{\expandafter\def\csname c#1\endcsname{{\mathcal #1}}}
\applytolist{calc}QWERTYUIOPLKJHGFDSAZXCVBNM;
\def\bbc#1{\expandafter\def\csname bb#1\endcsname{{\mathbb #1}}}
\applytolist{bbc}QWERTYUIOPLKJHGFDSAZXCVBNM;
\def\bfc#1{\expandafter\def\csname bf#1\endcsname{{\mathbf #1}}}
\applytolist{bfc}QWERTYUIOPLKJHGFDSAZXCVBNM;
\def\sfc#1{\expandafter\def\csname s#1\endcsname{{\sf #1}}}
\applytolist{sfc}QWERTYUIOPLKJHGFDSAZXCVBNM;
\def\fc#1{\expandafter\def\csname f#1\endcsname{{\mathfrak #1}}}
\applytolist{fc}QWERTYUIOPLKJHGFDSAZXCVBNM;

 \title{Anosov magnetic flows on surfaces}
 
 \author{James Marshall Reber and Yumin Shen}
\begin{document}

\begin{abstract}
Using the quotient bundle introduced by Wojtkowski, we give necessary and sufficient conditions for a magnetic flow on a closed, oriented surface to be Anosov. 
\end{abstract}
\maketitle

\section{Introduction} \label{sec:intro}
Let $(M,g)$ be a closed surface, and let $\phi^t : SM \rightarrow SM$ be a smooth flow on the unit tangent bundle. Denoting the infinitesimal generator of the flow $\phi^t$ by $X$ and equipping $SM$ with the Sasaki metric, we recall that the flow $\phi^t$ is \emph{Anosov} if $X$ is nowhere vanishing and there are $\phi^t$-invariant bundles $E^+$ and $E^-$ along with $c,d > 0$ so that $TSM = E^+ \oplus \R X \oplus E^-$ and $\|d_v \phi^{\pm t}(\xi)\| \leq d e^{-ct} \|\xi\|$ for all $\xi \in E^\pm(v)$ and $t \geq 0$.
Such flows can be seen as generalizations of geodesic flows of metrics with everywhere negative sectional curvature \cite{anosov}. In particular, they share many nice dynamical properties with a negatively curved geodesic flow; the flow is mixing, there is a dense orbit, and periodic points are dense.  It is of interest, then, to determine when a flow is Anosov.
% This is an abstract flow, not a geodesic flow.

For geodesic flows, it is shown in \cite{eberlein, eberlein2, klingenberg1974riemannian} that a geodesic flow is Anosov if and only if the metric is without conjugate points and there does not exist a non-trivial Jacobi field whose perpendicular component is bounded.
%\footnote{See \cite{hao} for a survey of this result in the case where $M$ is a surface.} 
In particular, this can be used to show that if the surface $M$ has non-positive curvature, then the geodesic flow is Anosov if and only if every geodesic in $M$ passes through a region with strictly negative curvature. In particular, this corollary was used to construct an example of a metric whose geodesic flow was Anosov and whose curvature had some flat regions. By the structural stability of Anosov flows, this allows for an example of an Anosov geodesic flow whose metric had some regions with positive curvature, allowing for examples of so-called, ``Anosov metrics,'' which are not even non-positively curved.\footnote{This corollary also has been used to construct examples of Anosov flows coming from a physical context, see \cite{HM, kourganoff2016anosov}.}

The goal of this paper is to explore a generalization of the work in \cite{eberlein, eberlein2, klingenberg1974riemannian} to the setting where $M$ is a closed, oriented surface, and the flow $\phi^t$ is a magnetic flow. Recall that a \emph{magnetic system} on $M$ is a pair $(g,b)$, where $g$ is a Riemannian metric and $b$ is a smooth function on $M$. The function $b$ is referred to as the \emph{magnetic intensity}. Since $M$ is oriented, there exists a complex structure $\i$ on $TM$, where $\i v$ is the vector rotated by $\pi/2$ according to the orientation of $M$. A curve $\gamma : \R \rightarrow M$ is a \emph{magnetic geodesic} if it satisfies the second-order differential equation
\begin{equation} \label{eqn:magnetic_defn} \frac{D \dot{\gamma}}{dt}(t) =  b(\gamma(t)) \i \dot{\gamma}(t), \tag{M}\end{equation}
where $D/dt$ is the covariant derivative along $\gamma$ induced by the Levi-Civita connection of $g$. This equation dictates the motion of a charged particle on $M$ under the influence of the magnetic intensity. Writing $\gamma_v$ for the unique magnetic geodesic satisfying $(\gamma_v(0), \dot{\gamma}_v(0)) = v$ with $v \in TM$, we have the \emph{magnetic flow} associated to the magnetic system $(g,b)$ on $TM$, which is given by $\smash{\phi_{g,b}^t}(v) \coloneqq (\gamma_v(t), \dot{\gamma}_v(t))$. 
Note that the magnetic flow associated to the magnetic system $(g,0)$ is the geodesic flow associated to the metric $g$, and as a result the magnetic flow can be seen as a generalization of the geodesic flow. The dynamics of magnetic systems have been studied for many decades; we point the reader to \cite{AS, arnold1346some, G} for some historical context.

An easy calculation using \eqref{eqn:magnetic_defn} shows that $SM$ is an invariant set for the magnetic flow. However, note that the dynamics of the magnetic flow are on $SM$ are not representative for other speeds, since if $\gamma$ is a unit speed magnetic geodesic for the system $(g,b)$ and $\lambda \neq 0$, then the curve $t \mapsto \gamma(\lambda t)$ is a magnetic geodesic with speed $\lambda$ for the system $(g, \lambda b)$, where here we are keeping track of orientation with the speed. Regardless, if we rescale the magnetic intensity appropriately, then we can view the magnetic flow with respect to other speeds as living on the unit tangent bundle. With this in mind, we study the restriction $\smash{\phi_{g,b}^t} : SM \rightarrow SM$.

The \emph{magnetic curvature} associated to the magnetic system $(g,b)$ is the function $\K^{g,b} : SM \rightarrow \R$ given by $\K^{g,b}(x,v) \coloneqq K^g(x) - d_xb(\i v) + b^2(x),$ where $K^g$ is the Gaussian curvature. This was first defined in \cite{paternain1996anosov}, and has since been used as a tool to measure dynamic and geometric information related to the magnetic system.\footnote{For more details on magnetic curvature, we direct the reader to \cite{paternain1996anosov,  wojtkowski, burnspaternain, valerio}.} For example, it is shown in \cite{wojtkowski} that if $\K^{g,b} < 0$, then the magnetic flow is Anosov.\footnote{See \cite{gouda,grognet, IVJ} for related results.} Our interest in the magnetic curvature is how it arises in the Jacobi equations for magnetic systems. 

A \emph{Jacobi field} for the magnetic system $(g,b)$ is the variational vector field along a unit speed magnetic geodesic associated to a variation of unit speed magnetic geodesics. This is a useful tool for understanding the infinitesimal behavior of the magnetic flow, and as a result we will need it in our analysis of Anosov magnetic flows. It was shown in \cite{paternain1996anosov} that if $J$ is a vector field along a magnetic geodesic $\gamma$, then $J$ is a Jacobi field if and only if it satisfies the Jacobi equations:
\begin{align*}
%\label{eqn:jacobi_defn0}
\nonumber 
(J^\top)^{\cdot}(t) &= b(\gamma_v(t)) J^\perp(t), \\  (J^\perp)^{\cdot \cdot}(t) &+ \K^{g,b}(\gamma_v(t), \dot{\gamma}_v(t)) J^\perp(t) = 0, \label{eqn:jacobi_defn} \tag{J}\end{align*}
where $J^\top(t) \coloneqq \langle J(t)$, $\dot{\gamma}_v(t) \rangle$ and $J^\perp(t) \coloneqq \langle J(t), \i \dot{\gamma}_v(t) \rangle$. We note that the perpendicular component of a Jacobi field behaves exactly like the perpendicular component of a Jacobi field in the Riemannian setting, except with the magnetic curvature replacing the Gaussian curvature. If a real valued function $J^\perp$ satisfies \eqref{eqn:jacobi_defn}, then we call it a \emph{perpendicular Jacobi field}.

Similar to the Riemannian setting, one can define a notion of conjugate points in terms of a magnetic exponential. For simplicity of presentation, we reframe this definition to be in terms of perpendicular Jacobi fields; the equivalence of these two concepts follows from \cite[Proposition 2.1.1]{herreros}. Given two points $x, p \in M$, we say that a point $x$ is \emph{conjugate} to $p$ if there is a magnetic geodesic $\gamma : \R \rightarrow M$, a $t \in \R$, and a perpendicular Jacobi field $J$ along $\gamma$ so that $\gamma(0) = p$, $\gamma(t) = x$, and $J^\perp(0) =  0 = J^\perp(t)$. As a result, we say that a magnetic system $(g,b)$ is \emph{without conjugate points} if for all $v \in SM$ and all perpendicular Jacobi fields $J^\perp$ along $\gamma_v$, we have $J^\perp(0) = 0$ implies $J^\perp(t) \neq 0$ for all $t \neq 0$. 

A necessary condition for the geodesic flow to be Anosov is for the metric to be without conjugate points \cite{klingenberg1974riemannian, mane}. If a metric $g$ is without conjugate points, then one can construct $\phi_g^t$-invariant subbundles $E^\pm \subseteq TTM$ which are referred to as the \emph{Green bundles} \cite{green}.\footnote{See also \cite[Proposition A]{contreras}.} One can then interpret the result in \cite{eberlein} as saying that the Green bundles restricted to $TSM$ form an Anosov splitting for the geodesic flow. Thus, our first step in understanding necessary and sufficient conditions for a magnetic flow to be Anosov is to understand how the magnetic flow being Anosov relates to being without conjugate points.

\begin{theorem}\label{thm:conjugate}
    Let $M$ be a closed, oriented surface, and let $(g,b)$ be a magnetic system on $M$. If $\smash{\phi_{g,b}^t} : SM \rightarrow SM$ is Anosov, then $(g,b)$ is without conjugate points.
\end{theorem} 

Using Theorem \ref{thm:conjugate}, we deduce two well-known conditions that the magnetic system must satisfy in order for the magnetic flow to be Anosov. We first describe a restriction on the magnitude of $b$. If the magnetic flow $\smash{\phi_{g,b}^t} : SM \rightarrow SM$ is Anosov and $\mu$ is the Liouville measure on $SM$, then \cite[Theorem A]{IVJ} and Theorem \ref{thm:conjugate} implies that we have 
\[ \int_{SM} \K^{g,b}(v) d\mu(v) \leq 0,\] with equality if and only if the magnetic curvature is zero. Using Stokes' theorem along with the Gauss-Bonnet theorem, we can rewrite the above as
\[ \int_M b^2(x) d\nu(x) \leq - 2\pi \chi(M),\]
where $\chi(M)$ is the Euler characteristic of $M$ and $\nu$ is the measure coming from the area form on $M$. 

As discussed in \cite{IVJ}, if the magnetic curvature is zero, then either the magnetic flow corresponds to the geodesic flow on a flat torus, or it corresponds to the horocycle flow of a hyperbolic metric. Neither of these flows are Anosov, and thus the inequality must be strict. Alternatively, one can use the structural stability of Anosov flows to deduce that the set of Anosov flows is open, and hence equality cannot occur. In either case, this inequality shows that the average of $b^2$ cannot be too big if the corresponding magnetic flow is Anosov, and we are able to recover \cite[Theorem 1.3]{paternain1996anosov}.

\begin{corollary} \label{cor:restriction1}
    Let $M$ be a closed, oriented surface, and let $(g,b)$ be a magnetic system on $M$. If $\smash{\phi_{g,b}^t} : SM \rightarrow SM$ is Anosov, then 
    \[ \int_M b^2(x) d\nu(x) < - 2\pi \chi(M).\]
    Moreover, if $b$ is not identically zero, then for $\lambda \in \R$ we have that if $\phi_{g,\lambda b}^t : SM \rightarrow SM$ is Anosov, then
    \[ \lambda^2 < \frac{- 2\pi \chi(M)}{\int_M b^2(x) d\nu(x)}.\]
\end{corollary}

\begin{rem}
As observed in \cite[Section 2]{burnspaternain}, the inequality with respect to $\lambda$ is not as sharp as the one given in \cite[Theorem B]{burnspaternain} unless $g$ has constant negative curvature.
\end{rem}

It follow from
\cite{plante1972anosov} that if the unit tangent bundle of a surface admits an Anosov flow, then the fundamental group of the unit tangent bundle must have exponential growth. Since the fundamental groups of the unit tangent bundles of the sphere and torus have polynomial growth, this implies that there cannot be an Anosov magnetic flow on the unit tangent bundle of either. We are able to recover this topological restriction using the strict inequality in Corollary \ref{cor:restriction1} (see also \cite[Corollaries A and B]{IVJ}).
%Since there is a strict inequality with $b$ in Corollary \ref{cor:restriction1}, we see that we are able to recover this topological restriction (c.f. \cite[Corollaries A and B]{IVJ}).

\begin{corollary}
    Let $M$ be a closed, oriented surface. If there exists a magnetic system $(g,b)$ on $M$ so that $\smash{\phi_{g,b}^t} : SM \rightarrow SM$ is Anosov, then the genus of $M$ is at least two.
\end{corollary}

Using the examples constructed in \cite[Section 7]{burnspaternain}, we observe that neither of the above conditions are sufficient for the magnetic flow to be Anosov. In order to find sufficient conditions, as well as prove Theorem \ref{thm:conjugate}, we introduce the quotient bundle and the quotient flow described in \cite{IVJ} (see also \cite[Section 2]{contreras}, \cite[Section 5]{wojtkowski}, and \cite[Section 5]{merry}). These will be our main tools for understanding Anosov magnetic flows.

Let $X^b$ be the infinitesimal generator of the flow associated to $\smash{\phi_{g,b}^t}$ on $SM$. The \emph{quotient bundle} over $SM$ associated to the magnetic system $(g,b)$ is the bundle with fibers given by $Q^b(v) \coloneqq T_vSM/ \R X^b(v)$. We equip it with the quotient norm coming from the Sasaki metric.
The differential of the magnetic flow 
induces a flow on the quotient bundle, which we refer to as the \emph{quotient flow} and denote by $\Phi^t : Q^b \rightarrow Q^b$. 
We say that the quotient flow is \emph{Anosov} if  there is a $\Phi^t$-invariant splitting $Q^b = \hat{E}^{+} \oplus \hat{E}^{-}$ and constants $c, d > 0$ so that for all $v \in SM$ we have
$\| \Phi_{v}^{\pm t}(\xi)\| \leq d \smash{e^{-ct}} \Vert\xi\Vert \text{ for all } t \geq 0 \text{ and } \xi \in \smash{\hat{E}^{\pm}(v)}.$

We will show in Section \ref{sec:stableunstable} that if $(g,b)$ is without conjugate points, then, as in the geodesic setting, there are two $\Phi^t$-invariant subbundles $\hat{E}^\pm \subseteq Q^b$, which we will refer to as the \emph{Green bundles} (see also \cite[Lemma 4.4]{IVJ} and \cite[Corollary 2.3 (b)]{contreras}). Moreover, as shown in \cite[Appendix A]{IVJ}, these bundles are the candidates for the Anosov splitting of $\Phi^t$ in the case where the magnetic curvature is everywhere strictly negative. Our next result shows that this is always the case (compare with \cite[Theorem 3.6]{eberlein} and \cite[Theorem C]{contreras}).

\begin{theorem}\label{thm:main}Let $M$ be a closed, oriented surface, let $(g,b)$ be a magnetic system without conjugate points, and let $\hat{E}^\pm \subseteq Q^b$ be the Green bundles described above. The following are equivalent.
\begin{enumerate}[(1).]
\item \label{thm:1} We have that $\Phi^t$ is Anosov, and the Green bundles form an Anosov splitting.
\item \label{thm:2} For every $v \in SM$, we have $\hat{E}^{+}(v)\cap \hat{E}^{-}(v)=\{0\}$. 
\item \label{thm:3} There does not exist a non-trivial bounded perpendicular Jacobi field.
\end{enumerate}
\end{theorem}

We finish by observing two corollaries of Theorem \ref{thm:main}. First, \cite[Proposition 5.1]{wojtkowski} tells us that $\Phi^t$ is Anosov if and only if $\phi_{g,b}^t$ is Anosov. As a consequence of this and Theorem \ref{thm:main}, we are able to recover the following (see also \cite[Theorem C]{contreras}).

\begin{corollary} \label{cor:main}
        Let $M$ be a closed, oriented surface, and let $(g,b)$ be a magnetic system without conjugate points. The magnetic flow $\smash{\phi_{g,b}^t} : SM \rightarrow SM$ is Anosov if and only if there does not exist a non-trivial bounded perpendicular Jacobi field.
\end{corollary}

Using Theorem \ref{thm:main} and Corollary \ref{cor:main}, we then establish the following magnetic version of \cite[Corollary 3.6]{eberlein}, which further highlights the importance of magnetic curvature for the dynamics of the magnetic system.

\begin{corollary} \label{cor:main2}
    Let $M$ be a closed, oriented surface, and let $(g,b)$ be a magnetic system such that $\K^{g,b} \leq 0$. The magnetic flow $\smash{\phi_{g,b}^t} : SM \rightarrow SM$ is Anosov if and only if for every $v \in SM$, \hbox{$\{t \in \R \ | \ K^{g,b}(\gamma_v(t), \dot{\gamma}_v(t)) < 0\} \neq \varnothing$}.
\end{corollary}

The organization of the paper is as follows. In Section \ref{sec:twistedconnector}, we discuss the construction of the quotient bundle in the setting where $M$ is a surface. In particular, we observe a correspondence between perpendicular Jacobi fields and the quotient bundle. In Section \ref{sec:conjugateproof}, we prove Theorem \ref{thm:conjugate} using the quotient bundle. In Section \ref{sec:stableunstable}, we recover \cite[Lemma 4.4]{IVJ} in the surface case, showing that if the system is without conjugate points, then we have the Green bundles described in \cite{eberlein}. In Section \ref{sec:comparisons}, we recall some general comparison theory for abstract Riccati equations. Finally, in Sections \ref{sec:stability} and \ref{sec:main}, we prove Theorem \ref{thm:main} and Corollary \ref{cor:main2} following the techniques in \cite{eberlein} and \cite{hao}. 

\subsection*{Acknowledgments}
We would like to thank Andrey Gogolev and Valerio Assenza for reading an early draft of the paper. We would like to thank Ivo Terek for his comments and suggestions throughout. We would also like to thank the OSU Cycle program for providing this undergraduate research opportunity.

\section{The quotient bundle}  \label{sec:twistedconnector}

For the remainder of the paper, let $M$ be a closed, oriented surface, and let $(g,b)$ be a magnetic system on $M$. The goal of this section is to briefly review the construction of the quotient bundle from \cite{IVJ} in the setting of a surface and explain how it is connected to the perpendicular components of Jacobi fields. 

Let $\xi \in T_vSM$, and consider the curve adapted to $\xi$ given by $t\mapsto(\alpha_{\xi}(t),\beta_{\xi}(t))$, where $\alpha_\xi(t)$ is a curve on the manifold and $\beta_{\xi}(t)$ is a vector field along $\alpha_\xi(t)$.
The (Levi-Civita) \emph{connector} is the map $K_{v} : T_{v} TM \rightarrow T_{\pi(v)}M,$ given by $\smash{K_{v}(\xi)} \coloneqq \smash{\nabla_{\dot{\alpha}_\xi(0)} \beta_\xi(0)}.$
Using the horizontal lift, it is easy to see that $T_vTM = \ker(K_v) \oplus \ker(d_v\pi)$, the restriction $K_v : \ker(d_v\pi) \rightarrow T_{\pi(v)}M$ is an isomorphism, and the restriction $d_v\pi : \ker(K_v) \rightarrow T_{\pi(v)}M$ is an isomorphism (see \cite[Lemma 1.15]{paternain2012geodesic}). 

Notice that we have $T_vSM = \{\xi \in T_vTM \ | \ \langle K_v(\xi), v \rangle = 0\}$ \cite[Exercise 1.28]{paternain2012geodesic}. 
If we quotient the fibers of $T_vSM$ by $X^b(v)$, then we get the quotient bundle described in Section \ref{sec:intro}. Utilizing the isomorphism $\xi + \R X^b(v) \mapsto \xi - \langle d_v\pi(\xi), v \rangle X^b(v)$, we see that\ without loss of generality we can identify 
\begin{equation} \label{eqn:quotient} \tag{Q} Q^b(v) = \{\xi \in T_vTM \ | \ \langle K_v(\xi), v \rangle = 0 = \langle d_v\pi(\xi), v \rangle \}.\end{equation}
The \emph{Sasaki metric} on $TM$ is given by $\|\xi\|^2 = \|K_v(\xi)\|^2 + \|d_v\pi(\xi)\|^2.$
Note that this induces norms on $TSM$ and $Q^b(v)$. Since $M$ is compact, we observe that any norm is equivalent to the quotient norm on $Q^b(v)$ coming from $TSM$. In particular, we can prove Theorem \ref{thm:main} using the norm induced by the Sasaki metric on $Q^b(v)$, as changing norm only changes the constant $d$.

Given $v \in SM$, we denote the space of Jacobi fields along $\gamma_v$ by $\J^b(v)$. Since Jacobi fields are given by a second order differential equation, we see that the footprint map and the connector give rise to an isomorphism between $\J^b(v)$ and $T_vSM$ via $(J(0), \dot{J}(0)) = (d_v\pi(\xi), K_v(\xi))$. Under this identification, we see that we have the following.

\begin{lemma} \label{lem:correspondence}
    We have that $Q^b(v) \cong \J^{b,0}(v) \coloneqq \{ J \in \J^b(v) \ | \ J^\top(0) = 0\}$.
\end{lemma}

For the remainder of the paper, we write $J_\xi$ for the Jacobi field in $\J^{b,0}(v)$ identified with $\xi \in Q^b(v)$. Notice that $d_v\pi(\xi) = J_\xi^\perp(0) \i v$, while $K_v(\xi) = \smash{\frac{D J_\xi}{dt}}(0) = [(J_\xi^\perp)^{\cdot}(0) + b(\pi(v)) J^\top(0)] \i v = (J_\xi^\perp)^{\cdot}(0) \i v$, where here we are using the fact that $J_\xi \in \J^{b,0}(v)$. We also have $\smash{d_{\phi_{g,b}^t(v)}\pi} \circ \Phi_v^t(\xi) = J_\xi^\perp(t)\i v$ and $\smash{K_{\phi_{g,b}^t(v)}} \circ \Phi_v^t(\xi) = (J_\xi^\perp)^{\cdot}(t)\i v$. Thus, the quotient bundle allows us to analyze the quotient flow in terms of perpendicular Jacobi fields.

To finish, we note that we can introduce a symplectic form on $Q^b(v)$ as follows. Let $J_\xi, J_\eta \in \J^{b,0}(v)$, and consider the Wronskian of the perpendicular Jacobi fields:
\begin{equation} \label{eqn:wronskian} \tag{W} W(J_\xi^\perp, J_\eta^\perp) \coloneqq (J_\xi^\perp)^{\cdot}(0) J_\eta^\perp(0) - (J_\eta^\perp)^{\cdot}(0) J_\xi^\perp(0).  \end{equation}
Since these satisfy \eqref{eqn:jacobi_defn}, it is easy to see that we can replace $0$ with $t$ in \eqref{eqn:wronskian}. Furthermore, this defines a symplectic form $\omega^b$ on $Q^b$ by setting $\omega^b_v(\xi, \eta) \coloneqq W(J_\xi^\perp, J_\eta^\perp)$ (see also \cite[Equation (4.1)]{IVJ}).  

To help with notation, we will omit the subscript $v$ in the connector, the differential of the footprint map, and the quotient flow when it is clear from context.

\section{Proof of Theorem \ref{thm:conjugate}}

\label{sec:conjugateproof}

Assume now that $(g,b)$ is a magnetic system on $M$ such that $\smash{\phi_{g,b}^t} : SM \rightarrow SM$ is Anosov. The goal of this section is to establish that $(g,b)$ is without conjugate points. 

As noted in Section \ref{sec:intro}, we have that there are $\Phi^t$-invariant subbundles $\hat{E}^\pm \subseteq Q^b$. Moreover, the dimension of $\hat{E}^+$ is one, so it is clearly Lagrangian with respect to the symplectic form $\omega^b$. Let 
\[\hat{V}(v) \coloneqq \{\xi \in T_vTM \ | \ d_v\pi(\xi) = 0 = \langle K_v(\xi), v \rangle\} = \ker(d_v\pi|_{Q^b(v)}).\]
The following lemma tells us that there is a smallest time at which a conjugate point can occur, and moreover this time is realized.

\begin{lemma} \label{lem:conjugate}
    Let $\gamma_v$ be a magnetic geodesic, and suppose that there is a $t_0 >0 $ and a perpendicular Jacobi field $J$ along $\gamma_v$ such that $J(0) = 0$ and $J(t_0) = 0$. Then there exists a $T(v) = T > 0$ so that for every perpendicular Jacobi field $J$ along $\gamma_v$ with $J(0) = 0$, we have $J(t) \neq 0$ for $t \in (0,T)$ and $J(T) = 0$. 
\end{lemma}

\begin{proof}
    We work with the case where $t_0 > 0$.
    Define the subspace $E(t) \coloneqq \Phi^{-t}(\hat{V}(\phi_{g,b}^t(v))),$ and define \hbox{$T \coloneqq \inf\{ t > 0 \ | \ E(t) \cap \hat{V}(v) \neq 0 \}.$}
    Since these subspaces are one-dimensional, we have that $E(t) \cap \hat{V}(v) \neq 0$ if and only if $E(t) = \hat{V}(v)$. In particular, this means that if $J$ is a perpendicular Jacobi field with $J(0) = 0$, and $E(t) \cap \hat{V}(v) \neq 0$, then $J(t) = 
    0$. 
    
    We first show that $T > 0$. Suppose for contradiction that $T = 0$. Let $(t_n)$ be a sequence such that $t_n \rightarrow 0$ as $n \rightarrow \infty$ and $E(t_n) = \hat{V}(v)$ for every $n$. In particular, taking a perpendicular Jacobi field $J$ with $J(0) = 0$, we have that $J(t_n) = 0$ for every $n$. By the mean value theorem, there is an $s_n \in (0, t_n)$ so that $\dot{J}(s_n) = 0$. Taking the limit as $n$ tends to infinity, we see that $\dot{J}(0) = 0$, and hence $J \equiv 0$, contradicting the fact that $\hat{V}(v) \neq 0$. 

    Next, let $t_n \rightarrow T$ be such that $E(t_n) = \hat{V}(v)$. If $J$ is a perpendicular Jacobi field with $J(0) = 0$, then $J(t_n) = 0$ for every $n$. Using the fact that $J$ is continuous, we deduce that $J(T) = 0$, hence $E(T) = \hat{V}(v)$. Furthermore, since $T$ is an infimum, we see that $J(t) \neq 0$ for $t \in (0,T)$. 
\end{proof}

Next, we recall that if the flow associated to $(g,b)$ is Anosov, then $\hat{E}^+(v)$ never intersects $\hat{V}(v)$.

\begin{lemma} \label{lem:conjugate2}
   If $\smash{\phi_{g,b}^t} : SM \rightarrow SM$ is Anosov, then $\hat{E}^+(v) \cap \hat{V}(v) = 0$ for all $v \in SM$. In other words, for every perpendicular Jacobi field $J^{\perp, +}$ corresponding to a vector in $\smash{\hat{E}^+(v)}$, we have $J^{\perp, +}(t) \neq 0$ for all $t \in \R$.
\end{lemma}

\begin{proof}
    Let $V : SM \rightarrow Q^b(v)$ be the smooth vector field satisfying $K_v(V(v)) = \i v$ and $d_v\pi(V(v)) = 0$; this is referred to as the \emph{vertical vector field} (see also \cite[Section 4]{merry}). In \cite[Lemma 4.1]{dairbekov2007entropy}, it was shown that $V(v) \notin E^+(v) \oplus \R X^b(v)$, where $E^+$ comes from the Anosov splitting of $\smash{\phi_{g,b}^t} : SM \rightarrow SM$. Using the identification described in Section \ref{sec:twistedconnector}, this implies that $V(v) \notin \hat{E}^+(v)$, and the result follows by observing that $\R V(v) = \hat{V}(v)$.
\end{proof}
With this, we are now able to prove Theorem \ref{thm:conjugate} using an argument similar to the Sturm separation theorem.

\begin{proof}[Proof of Theorem \ref{thm:conjugate}]
    Assume for contradiction that $\smash{\phi_{g,b}^t} : SM \rightarrow SM$ is Anosov and that there is a $v \in SM$ and a perpendicular Jacobi field $J^\perp$ along the magnetic geodesic $\gamma_v$ such that $J^\perp(0) = 0$ and $J^\perp(t_0) = 0$ for some $t_0 \neq 0$. The goal is to show that $\smash{\phi_{g,b}^t} : SM \rightarrow SM$ is not Anosov. Without loss of generality, we may assume that $t_0 > 0$, and by Lemma \ref{lem:conjugate} there is a $T > 0$ so that $J^\perp(0) = 0 = J^\perp(T)$ and $J^\perp(s) \neq 0$ for all $s \in (0,T)$.

    We see that the signs of $(J^\perp)^{\cdot}(0)$ and $(J^\perp)^{\cdot}(T)$ are different. Let $J^{+,\perp}$ be a perpendicular Jacobi field corresponding to a non-trivial vector $\xi \in \hat{E}^+(v)$. Notice that $W(J^{+, \perp}, J^\perp) = J^{+,\perp}(0) (J^\perp)^{\cdot}(0) = J^{+,\perp}(T) (J^\perp)^{\cdot}(T),$ hence the signs of $J^{+,\perp}(0)$ and $J^{+, \perp}(T)$ are different. As a consequence we must have that $J^{+,\perp}$ vanishes on $(0,T)$, contradicting Lemma \ref{lem:conjugate2}. 
\end{proof}

\begin{rem}
We observe that one can also deduce Theorem \ref{thm:conjugate} using \cite[Proposition 2.1.1]{herreros}, \cite[Corollary 1.18]{contreras}, and \cite{paternain1993anosov}. The advantage of the above proof is that we did not have to go through the construction of the Green bundles on the level of $TTM$.
\end{rem}

\section{Green bundles} \label{sec:stableunstable}

Assume now that the magnetic system $(g,b)$ is without conjugate points. The goal of this section is to review the construction of the Green bundles in \cite{IVJ}, and relate them to the stable and unstable bundles from \cite{eberlein}. 

Recall that magnetic geodesics are not flip invariant by the lack of homogeneity described in Section \ref{sec:intro}.
Since many of the arguments in \cite{eberlein} use flip invariance to relate Jacobi fields along $\gamma_v(t)$ to Jacobi fields along $\gamma_{-v}(t)$, we need a way around this problem. Fortunately, as noted in \cite[Lemma 3.2]{herreros}, there exists a correspondence between Jacobi fields and ``flipped'' Jacobi fields. Let $J^\perp$ be a perpendicular Jacobi field along the magnetic geodesic $\gamma_{v}$ with respect to the magnetic system $(g,b)$. Denote by $\eta_{-v}$ the magnetic geodesic for the system $(g,-b)$ given by $\eta_{-v}(t) \coloneqq \gamma_v(-t)$. Notice that $L^\perp(t) \coloneqq J^\perp(-t)$ satisfies the differential equation
\begin{equation*}
(L^\perp)^{\cdot \cdot}(t) + \K^{g,-b}_{-v}(\eta(t), \dot{\eta}(t)) L^\perp(t) = 0,\tag{$\text{J}^\prime$}    
\end{equation*}
so $L^\perp$ is a perpendicular Jacobi field along the magnetic geodesic $\eta_{-v}$ with respect to the magnetic system $(g,-b)$. Thus, changing the sign of the time parameter of a perpendicular Jacobi field not only changes the orientation, but it also changes the magnetic intensity. In particular, this correspondence proves the following. 

\begin{lemma} \label{lem:perpjacobflip}
    If the magnetic system $(g,b)$ is without conjugate points, then $(g,-b)$ is also without conjugate points.
\end{lemma}

Let $J^{\perp, r}$ be the perpendicular Jacobi field satisfying $J^{\perp, r}(0) = 1$ and $J^{\perp, r}(r) = 0$, and let $J^{\perp,z}$ be the perpendicular Jacobi field satisfying $J^{\perp,z}(0) = 0$ and $(J^{\perp,z})^{\cdot}(0) = 1$. Since we are assuming $(g,b)$ is without conjugate points, we have that $J^{\perp, r}$ and $J^{\perp, z}$ are uniquely defined. Similarly, for $\xi \in Q^b(v)$, let $\Psi_v^r(\xi) \in Q^b(v)$ be the unique vector satisfying $d\pi(\Psi_v^r(\xi)) = d\pi(\xi)$ and $d\pi \circ \Phi^r \circ \Psi_v^r(\xi) = 0.$ Using the correspondence in Lemma \ref{lem:correspondence}, we see that $\Psi_v^r(\xi)$ corresponds to the perpendicular Jacobi field $J^{\perp, r}_\xi(t) \coloneqq \langle d\pi(\xi), \i v \rangle J^{\perp, r}(t)$. As before, we will omit the subscript $v$ from $\Psi$ when it is clear from context.

In order to construct the Green bundles, we wish to take a limit of $J^{\perp, r}$ as $r$ tends to $\pm \infty$. In particular, assuming that the limits exist, let $\hat{E}^{\pm}(v)$ be the subbundles in $Q^b(v)$ corresponding to $\R J^{\perp, \pm \infty}$ under the correspondence in Lemma \ref{lem:correspondence}. We see that if the limits exist, then we can define the Green bundles in terms of $\Psi^r$ as $\hat{E}^{\pm}(v) \coloneqq \{ \xi \in Q^b(v) \ \big| \ \lim_{r \rightarrow \pm \infty} \Psi^r(\xi) = \xi\},$
which matches the construction in \cite[Lemma 4.4]{IVJ}. We now show that the limits exist.

\begin{lemma}\label{lem:Jinfty}
We have the following.
\begin{enumerate}[(1).]
\item \label{Jinfty1} The limit $\lim\limits_{r\rightarrow\pm\infty}(J^{\perp,r})^\cdot(0)$ exists.

\item \label{Jinfty2} If $J^{\perp,\pm\infty}$ is the solution to the perpendicular Jacobi equation with initial conditions  $J^{\perp,\pm\infty}(0)=1$ and $(J^{\perp,\pm\infty})^\cdot(0)=\lim\limits_{r\rightarrow\pm\infty}(J^{\perp,r})^\cdot(0)$, then
$J^{\perp,\pm\infty}(t)=\lim\limits_{r\rightarrow\pm\infty} J^{\perp,r}(t)$.

\item \label{Jinfty3} The perpendicular Jacobi field $J^{\perp,\pm\infty}(t)$ never vanishes.   
\end{enumerate}
\end{lemma}

\begin{proof}
We prove the case where $r$ tends to positive infinity, noting that the other case follows from Lemma \ref{lem:perpjacobflip}. In order to prove \eqref{Jinfty1}, we use the monotone convergence theorem. First, observe that for $t > 0$ we have 
\begin{equation}
\label{eqn:Jperp_r} 
J^{\perp,r}(t)=J^{\perp,z}(t)\int_{t}^{r}\frac{du}{(J^{\perp,z}(u))^2}. \tag{V}
\end{equation}
We can extend the solution to $t\in \mathbb R$ by existence and uniqueness of solutions to differential equations. Next, for any $r,t>0$ and $0<s<r$, observe that we have the relationship
\begin{equation*} \label{eqn:Jperp_r_id} J^{\perp,r}(t)-J^{\perp,s}(t)=J^{\perp,z}(t)\int_{s}^{r}\frac{du}{(J^{\perp,z}(u))^2}. 
%\tag{V2}
\end{equation*}
Taking the derivative at $t=0$ yields 
\begin{equation} \label{eqn:Jperp_r_id2}
(J^{\perp,r})^\cdot(0)-(J^{\perp,s})^\cdot(0)=\int_{s}^{r}\frac{1}{(J^{\perp,z}(u))^2}du. \tag{$\text{V}^\prime$}\end{equation}
Notice that this is an increasing function in $r$, so we are done if we can show it is bounded.

Given $q>0$, we have
\begin{equation}\label{eqn:1}
J^{\perp,-q}(t)=-\frac{J^{\perp,r}(-q)}{J^{\perp,z}(-q)}J^{\perp,z}(t)+J^{\perp,r}(t).\tag{E}
\end{equation}
Existence and uniqueness of differential equations ensures that equality (\ref{eqn:1}) holds regardless of the choice of $r>0$.
We now claim that $(J^{\perp,-1})^\cdot(0)>(J^{\perp,r})^\cdot(0).$
Indeed, we have the following relation by taking the derivative of equality (\ref{eqn:1}):$$(J^{\perp,-1})^\cdot(0)-(J^{\perp,r})^\cdot(0)=-\frac{J^{\perp,r}(-1)}{J^{\perp,z}(-1)}.$$ Thus, it suffices to show ${J^{\perp,r}(-1)}/{J^{\perp,z}(-1)}<0$. Using the fact that the system is without conjugate points along with the definition of $J^{\perp,r}$ and $J^{\perp,z}$,  we have $J^{\perp,r}(-1)>0$ and $J^{\perp,z}(-1)<0$, proving the claim.

Next, we note that \eqref{Jinfty2} follows from \eqref{Jinfty1} by smooth dependence of solutions to the ordinary differential equation.

Finally, we prove \eqref{Jinfty3}. Suppose for contradiction that there is an $r_0 \in \R$ so that $J^{\perp,+\infty}(r_0) = 0$ vanishes. Note that $r_0 \neq 0$ by definition. We now study what happens if $r_0 > 0$ or if $r_0 < 0$.

\begin{enumerate}[\text{Case }1.]
    \item Suppose $r_0 > 0$. By existence and uniqueness, we must have $J^{\perp,+\infty}(t)=J^{\perp,r_0}(t)$. However, for all fixed $t>0$, we have shown that $J^{\perp,r}(t)$ is strictly increasing in $r$, giving us a contradiction.

    \item Suppose $r_0<0$. By existence and uniqueness, we must have $J^{\perp,+\infty}(t)=J^{\perp,r_0}(t)$. 

Let $I=(r_0-\epsilon,r_0+\epsilon)\subset (-\infty,0)$. For $s\in I$, the map $s\mapsto (J^{\perp,s})^\cdot(0)$ is a diffeomorphism by smooth dependence, and existence and uniqueness. However, by definition, 
\[(J^{\perp,r_0})^\cdot(0)=(J^{\perp,+\infty})^\cdot(0)=\lim\limits_{r\rightarrow+\infty}(J^{\perp,r})^\cdot(0).\] 
For some $s\in I\subset (-\infty,0)$ and $r>0$ we must have $(J^{\perp,s})^\cdot(0)=(J^{\perp,r})^\cdot(0)$ by continuity, contradicting existence and uniqueness.\qedhere
\end{enumerate}

\end{proof}

We refer to $\hat{E}^+$ as the \emph{stable Green bundle}, and to $\hat{E}^-$ as the \emph{unstable Green bundle}. The goal now is to show that these bundles are $\Phi^t$-invariant. The first step is the following.

\begin{lemma} For every $v \in SM$, there exists a positive $c = c(v)$ and an $r_0 = r_0(v)$ so that if $\xi\in Q^b(v)$ satisfies $d \pi\circ \Phi^r(\xi)=0$ for some $r\ge r_0$, then $\lVert K(\xi)\rVert\le c\lVert d\pi(\xi)\rVert$.\label{lem:UBK}
\end{lemma}
\begin{proof}
By the correspondence between $\xi$ and $J^{\perp,r}_\xi$, we have $\|K(\xi)\| = \|d\pi(\xi)\| \cdot |(J_\xi^{\perp, r})^{\cdot}(0)|$.
By Lemma \ref{lem:Jinfty}, we see that for some $r>r_0$ we have $|(J_\xi^{\perp, r})^\cdot(0)|\le |(J_\xi^{\perp,+\infty})^\cdot(0)|+1$.	
Setting $c \coloneqq 1+|(J_\xi^{\perp,+\infty})^\cdot(0)|$, the result follows.
\end{proof}

With this, we can prove invariance.

\begin{prop}\label{prop:Invariance}
%(Invariance of S&U Bundles) 
For any $t\in\mathbb{R}$ and $v\in SM$, we have  $\Phi^t(\hat{E}^{\pm}(v))=\hat{E}^{\pm}(\smash{\phi_{g,b}^t}(v))$.  
\end{prop}

\begin{proof}
    We prove the proposition for $\hat E^{+}(v)$, as the proof for $\hat E^{-}(v)$ is similar. The goal is to show that if $\lim\limits_{s \rightarrow +\infty} \Psi^{s}(\xi) = \xi$, then $\lim\limits_{r \rightarrow +\infty} \Psi(\Phi^t(\xi)) = \Phi^t (\xi)$.
    %or in other words, 
    %\[ \lim_{r \rightarrow +\infty} \| \Psi^r(\Phi^t(\xi)) - \Phi^t(\xi)\| = 0.\]
    To that end, let's first observe that we can add and subtract $\Phi^t(\Psi^{t+r}(\xi))$ and use the triangle inequality to get 
    \[ \| \Psi^r(\Phi^t(\xi)) - \Phi^t(\xi)\| \leq \|\Phi^t(\xi - \Psi^{t+r}(\xi))\| + \|\Phi^t(\Psi^{t+r}(\xi)) - \Psi^r(\Phi^t(\xi))\|.\]
    Taking the limit as $r$ tends to $+\infty$ on both sides and using our assumption, we have
    \[ \lim_{r \rightarrow +\infty} \| \Psi^r(\Phi^t(\xi)) - \Phi^t(\xi)\| = \lim_{ r \rightarrow +\infty}\|\Phi^t(\Psi^{t+r}(\xi)) - \Psi^r(\Phi^t(\xi))\|.  \]

    Next, observe that applying Lemma \ref{lem:UBK} yields
    $$\lVert K\circ\Phi^t \circ \Psi^{t+r}(\xi) -K\circ \Psi^{r} \circ \Phi^t(\xi)\rVert\le c\lVert d\pi\circ\Phi^t \circ \Psi^{t+r}(\xi)-d\pi \circ \Psi^{r}\circ \Phi^t(\xi)\rVert.$$
    As a result, if we can prove that 
    \[ \lim_{r \rightarrow +\infty} \lVert d\pi\circ\Phi^t \circ \Psi^{t+r}(\xi)-d\pi \circ \Psi^{r}\circ \Phi^t(\xi)\rVert = 0,\]
    then the result follows. 
    To show this limit is zero, we start by observing that 
    \[d\pi\circ\Phi^t \circ \Psi^{t+r}(\xi)-d\pi \circ \Psi^{r} \circ \Phi^t(\xi) = d\pi \circ \Phi^t \left( \Psi^{t+r}(\xi) - \xi\right). \]
    Since $\lim\limits_{s \rightarrow +\infty} \Psi^s(\xi) = \xi$, we can use continuity of $d\pi$ and $\Phi^t$ for fixed $t$ along with the assumption to get 
    \[\lim_{r \rightarrow \infty} \lVert d\pi\circ\Phi^t \circ \Psi^{t+r}(\xi)-d\pi \circ \Psi^{r}\circ \Phi^t(\xi)\rVert = \lim_{r \rightarrow \infty} \lVert d\pi \circ \Phi^t \left( \Psi^{t+r}(\xi) - \xi\right)\rVert  = 0.\qedhere\]
\end{proof}

We finish this section by noting that the bundles $\hat{E}^\pm$ are Lagrangian with respect to the symplectic form in Section \ref{sec:stableunstable} since they are one-dimensional. Summarizing, we have the following (see also \cite[Lemma 4.4]{IVJ} and \cite[Corollary 2.3 (b)]{contreras}).

\begin{prop} \label{prop:green_bundles}
Let $M$ be a closed, oriented surface. If $(g,b)$ is a magnetic system without conjugate points, then the Green bundles $\hat{E}^{\pm}$ exist and are $\Phi^t$-invariant Lagrangian subbundles of $Q^b$.
\end{prop}

\section{Comparison theory} \label{sec:comparisons}

In this section, we use techniques from ordinary differential equations to give estimates between a perpendicular Jacobi field and its derivative.
%The perpendicular Jacobi equation in \eqref{eqn:jacobi_defn} is a second-order differential equation.
In order to use comparison theory, we analyze the associated first-order \emph{Riccati equation} along a magnetic geodesic $\gamma$:
\begin{equation} \label{eqn:R}u'(t)+u(t)^2+\mathbb K^{g,b}(\gamma(t), \dot{\gamma}(t))=0.\tag{R}
\end{equation}    
Notice that if $J^\perp(t)$ is a non-vanishing perpendicular Jacobi field along $\gamma$, then $(J^\perp)^\cdot(t)/J^\perp(t)$ is a solution to \eqref{eqn:R}. This gives us a connection between non-vanishing Jacobi fields and solutions to \eqref{eqn:R}.

The following standard theorem will be useful in the upcoming discussion.
\begin{teo}[{{\cite [Theorem 3.5]{lb.ode}}}]\label{1stcomp}
Assume that $f_1(x,y)<f_2(x,y)$ are both continuous on an open region $U\subset \mathbb R^2$. Suppose $y=\phi_i(x)$ for $i \in \{1,2\}$ are solutions to the differential equation $y'=f_i(x,y)$ with initial condition $y(x_0)=y_0$,
where $(x_0,y_0)\in U$ and $x \in (a,b)$. Then we have $\phi_1(x)<\phi_2(x)$ for $x\in(x_0,b)$ and $
    \phi_1(x)>\phi_2(x)$ for $x\in(a,x_0).$
\end{teo}
\iffalse
\begin{proof}
Let $\psi(x)=\Phi(x)-\phi(x),$
$\psi'(x_0)=F(x_0,\Phi(x_0))-f(x_0,\phi(x_0))>0.$
By continuity, $$\exists \delta>0, \forall x\in(x_0,x_0+\delta), \psi(x)>0.$$
Suppose there is $x_1>x_0$ such that $\Phi(x_1)\le \phi(x_1)$, we can find
$$\alpha=\min\{x\in (x_0+\delta, b), \psi(x)=0\}.$$
By definition, $\forall x\in(x_0,\alpha),\psi(x)>0$, so $\psi'(\alpha)\le 0$. However, $\psi'(\alpha)=F(\alpha,\Phi(\alpha))-f(\alpha,\phi(\alpha))>0$, contradiction.
Therefore, $\forall x\in (x_0,b), \Phi(x)>\phi(x)$.
\end{proof}
\fi
Since $M$ is a closed surface and $\K^{g,b}$ is continuous, let $k>0$ be such that $\mathbb K^{g,b}(v)>-k^2$ for all $v \in SM$.
%, promised by compactness. 
We will now compare solutions to the Riccati equation associated with $\mathbb K^{g,b}$ to solutions to the Riccati equation with constant negative magnetic curvature $-k^2$: 
\begin{equation*}\label{eqn:R1}
    v'(t)+v(t)^2-k^2=0. \tag{$\text{R}^\prime$}
\end{equation*}
The solution to \eqref{eqn:R1} is given explicitly by $v(t)=k\coth(kt-c),$
where $c\in\mathbb R$ is an arbitrary constant depending on the initial value. 
By construction, $u'=-u^2+k^2$ and $v'=-v^2-\mathbb K^{g,b}$.
We have $-y^2+k^2>-y^2-\mathbb K^{g,b}$, and we use Theorem \ref{1stcomp} to deduce the following.
\begin{prop}
\label{prop:comp1} 
Let $u(t)$ be a solution of \eqref{eqn:R}.
\begin{enumerate}[(1).]
\item \label{item:prop_comp1_1}If $u(t)$ is defined on $\mathbb{R}^{+}$, then $-k\le u(t)\le k\coth(kt)$ for all $t\in\mathbb R^+$. 
\item \label{item:prop_comp1_2} If $u(t)$ is defined on $\mathbb R$, then $|u(t)|\le k$ for all $t \in \R$.
\end{enumerate}
\end{prop}
%\begin{rem}
%The two cases have different maximal intervals of solution, which are $\mathbb R^+$ and $\mathbb R$ respectively.
%\end{rem}
\iffalse
\begin{proof}
We first prove the case where $u(t)$ is defined on $\mathbb R$.
We assume the conclusion is not true, and derive a contradiction using Theorem \ref{1stcomp}. If the conclusion is not true, then we must have some $t_0$ such that $ u(t_0) = k_0$, where $\vert k_0 \vert >k$. Without loss of generality, we assume that $k_0>k$. 
Thus, 
\[v(t)=k\coth\left(kt-kt_0+\dfrac{1}{2}\ln\dfrac{k_0+k}{k_0-k}\right)\] 
is the solution to \eqref{eqn:R1} with the initial value $v(t_0) = k_0$. By Theorem \ref{1stcomp}, $u(t)>v(t)$ when $t<t_0$, so $u(t)$ must blow up at 
$t=t_0-(\ln(k_0+k)-\ln(k_0-k))/(2k).$ This tells us that $u(t)$ cannot defined on $\mathbb R$, giving a contradiction.

The proof is nearly the same in the other case. We assume the conclusion is false, the only difference is that we must have some positive $t_0$ such that $u(t_0)=k_0>k$, and the rest of the proof works.
\end{proof}
\fi
With this, we can improve the bound in Lemma \ref{lem:UBK} for vectors in $\hat{E}^{\pm}(v)$.
\begin{prop}\label{col:comp1}
For all $v\in SM$ and $\xi\in \hat{E}^{\pm}(v)$, we have $\Vert K(\xi)\Vert\le k \Vert d\pi(\xi)\Vert$.
\end{prop}
\begin{proof}
Under the identification in Lemma \ref{lem:correspondence}, the perpendicular Jacobi field corresponding to $\xi\in \hat{E}^{\pm}(v)$ is $J^{\perp}_{\xi}(t)=\Vert d\pi\xi\Vert \hspace{0.1em} J^{\perp,\pm\infty}(t)$.
We observe that $J_\xi^{\perp}(t)$ never vanishes by Lemma \ref{lem:Jinfty}. Thus, we have that the function $u(t) \coloneqq (J_\xi^{\perp})^\cdot (t)/J^{\perp}_{\xi}(t)$ 
is a solution to \eqref{eqn:R} that is defined for all $t\in\mathbb{R}$. By Proposition \ref{prop:comp1} \eqref{item:prop_comp1_2}, we have $|u(t)| \le k ,$
which we can rewrite as $|(J_\xi^{\perp})^\cdot (t)|\le k| J^{\perp}_{\xi}(t)|$.
Plugging in $t=0$ yields the desired result. 
%$\Vert K(\xi)\Vert=|(J_\xi^{\perp})^\cdot (0)| \le k\Vert d\pi(\xi)\Vert. $
\end{proof}

\iffalse

\begin{prop}\label{prop:comp2}
If $u(t)$ is a solution of \eqref{eqn:R} defined on $\mathbb{R}^{+}$, then $-k
\le u(t)\le k\coth(kt)$ for all $t>0$.
\end{prop}
\begin{proof}\label{prop:comp3}
We assume the conclusion is not true. Then there is a $t_0>0$ such that $u(t_0)=k_0>k\coth(kt_0)$. The initial value problem to \eqref{eqn:R1} and with initial condition $v(t_0)=k_0$ is given explicitly by \[v(t)=k\coth\left(kt-kt_0+\dfrac{1}{2}\ln\dfrac{k_0+k}{k_0-k}\right).\] We see that it has to blow up at $t=t_0-\dfrac{1}{2k}\ln\dfrac{k_0+k}{k_0-k}>0.$
\end{proof}
\textcolor{red}{Bundle with Proposition 8.}
\fi
Finally, we note that a perpendicular Jacobi field that vanishes somewhere tends to infinity.

\begin{prop}\label{prop:comp4}
Let $v \in SM$. If $J^\perp(t)$ is a perpendicular Jacobi field along the magnetic geodesic $\gamma_{v}$ such that $J^\perp(0)=0$, then for all $R>0$, 
there exists $T=T(R, v)>0$ such that for all $t\ge T$, we have ${\vert J^\perp(t)\vert\ge R\vert (J^\perp)^\cdot(0)\vert}$.
\end{prop}
\begin{proof}
If $(J^\perp)^\cdot(0)=0$, then the perpendicular Jacobi field is trivial and we are done. By normalizing, we may assume without loss of generality that $(J^\perp)^\cdot(0)=1$, so $J^\perp(t)=J^{\perp,z}(t)$. With this in mind, we only need to prove that ${\vert J^{\perp,z}(t)\vert\ge R\text{ for all }R>0}$.
%
% BEGIN COMMENT
\iffalse
It suffices to show $\lim_{t\rightarrow+\infty}J^{\perp,z}(t)=+\infty$, which follows if we show $\lim_{t\rightarrow+\infty}{(J^{\perp,z}(t))^{-2}}=0$.
However, recall that we have shown the following relation in the proof of Lemma \ref{lem:Jinfty}: 
$$(J^{\perp,r})^\cdot(0)-(J^{\perp,s})^\cdot(0)=\int_{s}^{r}\frac{1}{(J^{\perp,z}(t))^2}dt.$$
\fi 
% END COMMENT
%
This follows from \eqref{eqn:Jperp_r_id2}, since  the limit as $r \rightarrow + \infty$ exists for every $s > 0$.
%, and hence $\lim_{t\rightarrow+\infty}{J^{\perp,z}(t)}=0$.
%exists as $r\rightarrow \infty$ exists, hence we have $\lim_{t\rightarrow+\infty}{(J^{\perp,z}(t))^{-2}}=0$.
\end{proof}

\section{Stability criteria} \label{sec:stability}

In this section, we show that the Green bundles $\hat E^{\pm}$ satisfy important properties which imply that these are the correct candidates for the Anosov splitting. The next proposition gives a sufficient condition for a vector in the quotient bundle to be in the (un)stable Green bundle.

\begin{prop}\label{prop:crt1}
Let $v\in SM$ and $\xi\in Q^b(v)$. If there exists an $A>0$ so that $\Vert d\pi\circ \Phi^{\pm t}(\xi)\Vert\le A$ for all $t\ge 0$, then $\xi\in \hat E^{\pm}(v)$.
\end{prop}
\begin{proof}
We prove the case for $\xi\in \hat E^{+}(v)$, as the proof for $\hat{E}^-(v)$ is similar. 
Our goal is to show that ${\lim_{r\rightarrow +\infty}\Psi^{r}(\xi)= \xi}$. By definition, $d\pi(\Psi^{r}(\xi))= d\pi(\xi)$ for all $r$, so the result follows once we show \hbox{$\lim_{r \rightarrow + \infty} K(\Psi^{r}(\xi))=  K(\xi)$}.

Suppose that for some $A > 0$ we have $\Vert d\pi\circ \Phi^t(\xi)\Vert\le A$ for all $t\ge0$. By construction, $\Psi^r(\xi)$ can be identified with the perpendicular Jacobi field $J_\xi^\perp(0)J^{\perp,r}(t)$, and hence $K\circ \Psi^r(\xi)=J_\xi^\perp(0)(J^{\perp,r})^\cdot(0)$. 
If we can show that $\lim_{r\rightarrow+\infty}[J_\xi^\perp(0)(J^{\perp,r})^\cdot(0)-(J_\xi^\perp)^\cdot(0)]=0$, then it follows that $\lim_{r \rightarrow +\infty} K(\Psi^r(\xi)) = K(\xi)$.  Notice that the inside of the limit is precisely the Wronskian $W(J_\xi^\perp, J^{\perp,r})$, and thus it suffices to show that $\lim_{r \rightarrow +\infty} W(J^\perp_\xi, J^{\perp, r}) = 0$.

%$=J_\xi^\perp(0)(J^{\perp,r})^\cdot(0)-(J_\xi^\perp)^\cdot(0)J^{\perp,r}(0)$.

Write $J_{r,\xi}^\perp(t)\coloneqq J_\xi^\perp(0)J^{\perp,r}(t)-J_\xi^\perp(t)$. Since $J_{r,\xi}^\perp(0) = 0$, there exists $C(r) \in \R$ so that $J_{r,\xi}^\perp(t) = C(r) J^{\perp,z}(t)$. Furthermore, $(J_{r,\xi}^\perp)^{\cdot}(0) = W(J_\xi^\perp, J^{\perp, r}) = C(r)$. 
%Thus, it suffices to show that $\lim_{r \rightarrow +\infty} C(r) = 0$. 
We now use this observation to get a bound on the Wronskian.
Since the magnetic system is without conjugate points, the following holds for every $t > 0$:
\[\lim_{r \rightarrow +\infty} \vert C(r)\vert=\bigg\vert J_\xi^\perp(0)\frac{J^{\perp,{+\infty}}(t)}{J^{\perp,z}(t)}-\frac{J_\xi^\perp(t)}{J^{\perp,z}(t)}\bigg\vert \le \bigg\vert A\int_t^{+\infty}\frac{du}{(J^{\perp,z}(u))^2}\bigg\vert+\bigg\vert\frac{A}{J^{\perp,z}(t)}\bigg\vert.\]
As noted in Proposition \ref{prop:comp1}, we have  $\lim_{t\rightarrow+\infty}J^{\perp,z}(t)=+\infty$. Hence, we can use the fact that $C(r)$ is independent of $t$ along with \eqref{eqn:Jperp_r} and the assumption to get that $\lim_{r \rightarrow +\infty} |C(r)| = 0$.
%\[ \lim_{r \rightarrow +\infty} \vert C(r)\vert \le\lim_{t\rightarrow\infty} \left\{\bigg\vert A\int_t^{+\infty}\frac{du}{(J^{\perp,z}(u))^2}\bigg\vert+\bigg\vert\frac{A}{J^{\perp,z}(t)}\bigg\vert\right\}=0.\]
%and therefore, $C(+\infty)=0$.\qedhere
%The result follows.
\end{proof}
\iffalse
\begin{proof}
We prove the case for $\xi\in \hat E^{+}(v)$; the proof of the unstable case is similar. 
Our goal is to show that ${\lim_{r\rightarrow +\infty}\Psi^{r}(\xi)= \xi}$. By definition, $d\pi(\Psi^{r}(\xi))= d\pi(\xi)$ for all $r$, so it suffices to prove that $\lim_{r \rightarrow + \infty} K(\Psi^{r}(\xi))=  K(\xi)$.

Suppose $\Vert d\pi\circ \Phi^t(\xi)\Vert\le A$ for  all $t\ge0$ and some $A>0$. 
By linearity of the Jacobi equation, the perpendicular Jacobi field $J^\perp_{\xi-\Psi^{r}(\xi)}$ is the unique perpendicular Jacobi field which satisfies:
$$J^\perp_{\xi-\Psi^{r}(\xi)}(0)=0\ \text{ and }\ J^\perp_{\xi-\Psi^{r}(\xi)}(r)=J^\perp_{\xi}(r).$$
Using Proposition \ref{prop:comp4}, we have
$$\Vert K(\xi-\Psi^{r}(\xi))\Vert=\vert(J^\perp_{\xi-\Psi^{r}(\xi)})^\cdot(0)\vert\le \frac{\vert J^\perp_{\xi-\Psi^{r}(\xi)}(r)\vert}{R}\le\frac{A}{R}.$$
Letting $r$ tend to $+\infty$, we see that $\Vert K(\xi-\Psi^{r}(\xi))\Vert$ tends to $0$.
\end{proof}
\fi
The next observation emphasizes the symmetry between the magnetic systems $(g,b)$ and $(g,-b)$. In particular, the correspondence described in Lemma \ref{lem:perpjacobflip} shows that boundedness of a perpendicular Jacobi field is flip-invariant. 
% Here I'm guessing we are using Lemma 3 implicitly, but double check. If we are not, then maybe cut Lemma 3, or find somewhere we are using it and say it explicitly (otherwise there's no reason to give it a label).
\begin{prop}\label{Flip_invariance_of_Boundedness}
If the magnetic system $(g,b)$ does not admit a non-trivial bounded perpendicular Jacobi field $J^\perp$ on any unit speed magnetic geodesic $\gamma$, then $(g,-b)$ does not admit a non-trivial bounded perpendicular Jacobi field $J^\perp$ on any unit speed magnetic geodesic $\gamma$. 
\end{prop}
%\vspace{-20pt}
%\begin{proof} 
%We only need to prove the implication due to the symmetry of this statement. 
%Suppose a nonzero bounded perpendicular Jacobi field $J^\perp(t)$ satisfies the Jacobi equation. It suffices to construct a bounded perpendicular Jacobi field for the reversed magnetic system $(M,g,-b)$.
%Notice that for the reverse magnetic geodesic, we can construct a Jacobi field by setting
%$$(L_{(x,-v)}^\perp)^{\cdot \cdot}(t) + \K^{g,-b}_{(x,-v)}(t) L_{(x,-v)}^\perp(t) = 0.$$
%With this, we have constructed a bounded perpendicular non-trivial Jacobi field.
%\end{proof}
%Next, let $J^{\perp, z}_v$ be the perpendicular Jacobi field along $\gamma_v$ that satisfies $J^{\perp,z}_v(0) = 0$ and $(J^{\perp, z}_v)^{\cdot}(0) = 1.$ 
The following three propositions aim to give a partial converse to Proposition \ref{prop:crt1}.
%in the form of Proposition \ref{prop:crt4}. 
We start by observing a technical result.

\begin{prop} \label{prop:crt2}
Consider the function $g(t)\coloneqq\inf  \{ \vert J^{\perp,z}_{v}(t)\vert \ | \ v \in SM\}$. For all $t>0$, we have $g(t)>0$.
\end{prop}
\begin{proof}
Fix $t>0$. We can choose a sequence of $v_n \in SM$ so that $\lim_{n\rightarrow\infty}\vert J^{\perp,z}_{v_n}(t)\vert =g(t).$
By compactness of $SM$, this sequence has a convergent subsequence, say $v_{n_k}\rightarrow v'$. Since $(g,b)$ is without conjugate points, we must have $g(t)=\vert J^{\perp,z}_{v'}(t)\vert >0$.
%and the result follows.
%otherwise $J^{\perp,z}_{v'}(0)=J^{\perp,z}_{v'}(t)=0$, and we will get $(J^{\perp,z}_{(x',v')})^\cdot(0)=0$ which contradicts with the definition of $J^{\perp,z}$. 
\end{proof}

Using the previous proposition, we can now give a bound on Jacobi fields assuming that $(g,b)$ does not admit a non-trivial bounded perpendicular Jacobi field.

\begin{prop}\label{prop:crt3}
Assume $(g,b)$ does not admit a non-trivial bounded perpendicular Jacobi field $J^\perp$. There exists a constant $A>0$, independent of the choice of magnetic geodesic, such that if $J^\perp$ is a non-trivial perpendicular Jacobi field on a unit speed magnetic geodesic $\gamma$ with $J^\perp(0)=0$, then ${\vert J^\perp(t)\vert\ge A\vert J^\perp(s)\vert\text{ for all }t\ge s\ge 1}$. 
\end{prop}
%\vspace{-15.3pt}
\begin{proof}
We prove this by contradiction. Assume that there is a sequence of perpendicular Jacobi fields $J^\perp_n(t)$ on unit speed magnetic geodesics $\gamma_n$ with $J_n^\perp(0)=0$, and a sequence of numbers $s_n$ and $t_n$ with $1\le s_n\le t_n$ such that $\vert J_n^\perp(t_n)\vert\le \frac{1}{n}\vert J_n^\perp(s_n)\vert$. For each $J_n^\perp(t)$, we can choose $u_n$ so that $\vert J_n^\perp(u_n)\vert=\sup_{0\le s\le t_n}\vert J_n^\perp(s)\vert$. 

We claim that the sequence $\{u_n\}$ has a positive lower bound. Indeed, if not, then we can find a subsequence $\{u_{n_k}\}$ which tends to zero as $k$ tends to infinity. After possibly refining the subsequence further so that $J_{n_k}^\perp$ converges to $J^\perp$, we observe that $\lim_{k\rightarrow\infty}\vert J^\perp_{n_k}(u_{n_k})\vert=J^\perp(0)=0.$ By definition,  $\vert J^\perp_{n}(u_n)\vert\ge \vert J^\perp_n(1)\vert\ge g(1).$
However, by Proposition \ref{prop:crt2}, 
$g(1)> 0$,
and thus we have a contradiction. Let $M > 0$ be the lower bound for the sequence $\{u_n\}$.

We now normalize the perpendicular Jacobi fields. For each $n$, let $D_n(t)\coloneqq J_n^\perp(u_n+t)/\vert J_n^\perp(u_n)\vert$ be a Jacobi field along $\gamma_n(u_n+t)$. Notice that $D_n(t)$ satisfies $\vert D_n(-u_n)\vert=0$, $\vert D_n(0)\vert=1,$ and $\vert D_n(t)\vert\le 1$ for all  $t\in[-u_n, t_n-u_n].$ Moreover, we have that
$$\vert D_n(t_n-u_n)\vert=\frac{\vert J_n^\perp(t_n)\vert}{\vert J_n^\perp(u_n)\vert}\le \frac{1}{n}\frac{\vert J_n^\perp(s_n)\vert}{\vert J_n^\perp(u_n)\vert}\le\frac{1}{n}.$$	
Using Proposition \ref{prop:comp1} \eqref{item:prop_comp1_1}, we see that
\vspace{-10pt}
$$\vert (D_n)^\cdot(0)\vert=\frac{\vert (J_n^\perp)^\cdot(u_n)\vert}{\vert J_n^\perp(u_n)\vert}\le k\coth(ku_n)\le  k\coth(k M).$$ 
Using compactness, we may assume that $D_n$ converges to $D$ by passing to a subsequence $D_{n_k}$. By continuity, $\vert D(0)\vert=1$, so $D\ne 0$. We now investigate the limiting behavior of $t_n-u_n$ and $u_n$ along further refinements of the subsequence for which these all of these terms converge.
\begin{enumerate}[$\text{Case }$1.]		
\item Suppose we have $\lim_{k\rightarrow \infty}(t_{n_k}-u_{n_k})= \ell \ge 0$ and $\lim_{k\rightarrow \infty}u_{n_k}=u\ge M$. Using continuity, we have $D(-u)=\lim_{k\rightarrow\infty}D_{n_k}(-u_{n_k})=0$ and $D(\ell)=\lim_{k\rightarrow\infty}D_{n_k}(t_{n_k}-u_{n_k})=0$. Thus we have a pair of conjugate points, which is a contradiction.
\item Suppose we have $\lim_{k\rightarrow \infty}(t_{n_k}-u_{n_k})= \ell \ge 0$ and 
$\lim_{k\rightarrow \infty}u_{n_k}=+\infty$.
Using continuity, we have $D(\ell)=\lim_{k\rightarrow\infty}D_{n_k}(t_{n_k}-u_{n_k})=0$ and $\lim_{t\rightarrow-\infty}D_{n_k}(-u_{n_k})=0$. We also have $\vert D(t)\vert\le 1$ for all $t\le \ell$, because $\vert D_{n_k}(t)\vert\le 1$ for all $t\in[-u_{n_k},t_{n_k}-u_{n_k}]$, and $-u_{n_k}$ tends to $ -\infty$. {However, applying Proposition \ref{prop:comp4} to the reversed Jacobi field,} we must have $\lim_{t\rightarrow-\infty}\vert D(t)\vert=+\infty$, which is a contradiction.
\item Suppose we have $\lim_{k\rightarrow \infty}(t_{n_k}-u_{n_k})=+\infty$ and 
$\lim_{k\rightarrow \infty}u_{n_k}=u\ge M$.
Using continuity, we have $\lim_{t\rightarrow+\infty}D(t)=\lim_{k\rightarrow\infty}D_{n_k}(t_{n_k}-u_{n_k})=0$ and $D(-u)=\lim_{t\rightarrow-\infty}D_{n_k}(-u_{n_k})=0$. We also have $\vert D(t)\vert\le 1$ for all $t\ge -u$ because $\vert D_{n_k}(t)\vert\le 1$ for all $t\in[-u_{n_k},t_{n_k}-u_{n_k}]$, and $t_{n_k}-u_{n_k}$ tends to $+\infty$. {However, applying Proposition \ref{prop:comp4} to the Jacobi field,} we must have $\lim_{t\rightarrow+\infty}\vert D(t)\vert=+\infty$, which is a contradiction.
\item Suppose we have $\lim_{k\rightarrow \infty}(t_{n_k}-u_{n_k})=+\infty$ and $\lim_{k\rightarrow \infty}u_{n_k}=+\infty$. By continuity, we have ${\vert D(t)\vert\le 1}$ for all $ t\in\mathbb{R}$, since ${\vert D_{n_k}(t)\vert\le 1}$ for all $t\in[-u_{n_k},t_{n_k}-u_{n_k}]$, and $-u_{n_k}$ tends to $-\infty$ while $t_{n_k}-u_{n_k}$ tends to $+\infty$. This gives us a bounded perpendicular Jacobi field, which is a contradiction.
\end{enumerate}
Since all of the cases give us a contradiction, the result follows.
\end{proof}

Finally, we have the tools to prove a partial converse to Proposition \ref{prop:crt1}.

\begin{prop}  \label{prop:crt4}
Assume $(g,b)$ does not admit a non-trivial bounded perpendicular Jacobi field. For $v \in SM$, we have that $\xi\in \hat E^{\pm}(v)$ if and only if $\Vert d\pi\circ \Phi^{\pm t}(\xi)\Vert$ is bounded above $\text{for all }t\ge 0$.
\end{prop}

\begin{proof}
By Proposition \ref{prop:crt1}, it suffices to prove the implication. We prove the case where ${\xi\in \hat E^{+}(v)}$, as the other case is similar. Applying Proposition \ref{prop:crt3}, we get that there exists a constant $A>0$ such that if $J^\perp$ is a non-trivial perpendicular Jacobi field along a unit speed magnetic geodesic $\gamma$ which vanishes at the origin, then ${\vert J^\perp(t)\vert\ge A\vert J^\perp(s)\vert\text{ for all }t\ge s\ge 1}$. Let $J^\perp(u)$ be a perpendicular Jacobi field corresponding to $\Psi^t(\xi)$. As pointed out in the proof of Lemma \ref{lem:perpjacobflip}, the function $L^{\perp}(t)\coloneqq J^\perp(-t)$ is a perpendicular Jacobi field along the magnetic geodesic $\gamma_{-v}$ for the magnetic system $(g,-b)$.  Let $u\ge 0$ be fixed. Notice that for $t\ge 1+u$, we have 
$\Vert d\pi\circ\Phi^u \circ \Psi^t(\xi)\Vert=\vert J^\perp(u)\vert=\vert L^\perp(-u)\vert=\vert L^\perp(-t+(t-u))\vert.$

Now consider the perpendicular Jacobi field $L_{-t}^\perp(s)\coloneqq L^\perp(s-t)$. This is a perpendicular Jacobi field along $\gamma_{-v}(s-t)$ which vanishes at the origin, so $\Vert d\pi\circ\Phi^u \circ \Psi^t(\xi)\Vert=\vert L_{-t}^\perp(t-u)\vert\le \vert L_{-t}^\perp(t)\vert/A.$ 
Working through the definitions, we see that $\vert L_{-t}^\perp(t)\vert = \Vert d\pi(\xi)\Vert$, 
%$$\vert L_{-t}^\perp(t)\vert=\vert L^\perp(0)\vert=\vert J^\perp(0)\vert=\Vert d\pi\circ \Psi^t(\xi)\Vert=\Vert d\pi(\xi)\Vert,$$
hence
$\Vert d\pi\circ\Phi^u \circ \Psi^t(\xi)\Vert\le \Vert d\pi(\xi)\Vert/A.$
Taking the limit as $t$ tends to $+\infty$, the result follows by continuity.
\end{proof}

\begin{rem} Note that we only use the conclusion of Proposition \ref{prop:crt3} to prove Proposition \ref{prop:crt4}. This observation will be useful in Section \ref{sec:main} for the proof of Corollary \ref{cor:main2}. \label{rem:crt3.5}
\end{rem}

The next two propositions say that, under certain assumptions, a vector in the stable Green bundle is contracted by an exponential factor, and a vector in the unstable Green bundle is expanded by an exponential factor.
\begin{prop}\label{prop:crt5}
Assume $(g,b)$ does not admit a non-trivial bounded perpendicular Jacobi field $J^\perp$ on a unit speed magnetic geodesic $\gamma$. Then 
for all $\varepsilon>0$, there is a $T>0$ so that for all $v \in SM$ and for all $\xi\in \hat E^{\pm}(v)$, we have $\Vert \Phi^{\pm t}(\xi)\Vert\le\varepsilon\Vert \xi\Vert$ as long as $t\ge T$.
\end{prop}
\begin{proof}
We only prove the statement for $\xi\in \hat E^{+}(v)$, since the other case is similar. We assume the conclusion is not true and construct a non-trivial bounded perpendicular Jacobi field explicitly.

If the conclusion is not true, then we can find sequences $t_n$ tending to $+\infty$, $v_n \in SM$, and ${\xi_n\in \hat{E}^{+}(v_n)}$ with $\Vert\xi_n\Vert=1$ so that $\Vert\Phi^{t_n}(\xi_n)\Vert\ge\epsilon$.
%Note that the sequence $(\xi_n)$ lives in a compact set. 
Since ${\xi_n\in \hat{E}^{+}(v_n)}$, the perpendicular Jacobi field corresponding to $\xi_n$ never vanishes by Lemma \ref{lem:Jinfty}. Using Proposition \ref{col:comp1},
there exists $\alpha>0$ so that
\[\Vert\Phi^{t_n}(\xi_n)\Vert\le (1+\alpha^2)^{\frac{1}{2}}\left\Vert d\pi\circ\Phi^{t_n}(\xi_n)\right\Vert.\]
Proposition \ref{prop:crt4} then yields 
$$\left\Vert d\pi\circ\Phi^{t_n}(\xi_n)\right\Vert\le \frac{1}{A} \Vert d\pi(\xi_n)\Vert \le \frac{1}{A}\Vert\xi_n\Vert=\frac{1}{A}.$$
Thus, we can find a number $B>0$ so that for every $n$, $\Vert\Phi^{t+t_n}(\xi_n)\Vert\le B$ for $t+t_n \ge 0$.
By compactness, the sequence $\{\Phi^{t_n}(\xi_n)\}$ has a convergent subsequence, say $\{\Phi^{t_{n_k}}(\xi_{n_k}) \}$. Let $\psi \in T_vSM$ be the limit of this subsequence. Continuity then implies that ${\Vert \psi\Vert\ge \epsilon>0}$, which means $\psi$ corresponds to a non-trivial bounded perpendicular Jacobi field,
%and $\Vert\Phi^t(\psi)\Vert\le B$ for all $t\in\mathbb R$. Thus $\psi$ corresponds to a bounded perpendicular Jacobi field, 
giving a contradiction.
\end{proof}

\begin{prop} \label{prop:crt6}
Assume that $(g,b)$ does not admit a non-trivial bounded perpendicular Jacobi field on a unit speed magnetic geodesic $\gamma(t)$. Let $h(t)\coloneqq\sup\{ \Vert\Phi^t\xi\Vert \ \big| \ \xi\in \smash{\hat {E}^{+,b}}, \Vert\xi\Vert=1 \}.$ 
Then there exists $c,d>0$ so that for all $t\ge 0$, we have $h(t)\le de^{-ct}$.
\end{prop}
\begin{proof}
First, by Proposition \ref{prop:crt5}, we know $h(t)$ is bounded above for $t \geq 0$ since $\xi\in \smash{\hat {E}^{+,b}}$ and $\Vert\xi\Vert=1$. Next, observe that $\Vert\Phi^{s+t}\xi\Vert\le h(t)\Vert\Phi^{s}\xi\Vert\le h(t)h(s)$. Since this holds for all $\xi$, we have $h(s+t)\le h(s)h(t)$.
Fix $0<\epsilon<\smash{\frac{1}{2}}$, we can find $T>0$ such that for all $t\ge T$, $h(t)<\epsilon$. By iterating, we have $h(nT)\le h(T)^n<\epsilon^n$. For every $t>T$, we can find $n$ so that $nT\le t< \epsilon^{-1}nT$. This implies ${h(t)\le\left(h(t/n)\right)^n<\epsilon^n=e^{nln\epsilon}<e^{(t/T)\epsilon ln\epsilon}}$. Thus, letting $c\coloneqq -\epsilon ln\epsilon/T$, and $d\coloneqq \max\{1, \sup\{ \epsilon^{(t/T)\epsilon}h(t)\ |\ t\in [0,T]\}\}$, we have
$h(t)<de^{-ct}$ for all $t\ge 0$.
%First, by Proposition \ref{prop:crt4}, we know $\Vert d\pi\circ \Phi^t(\xi)\Vert$ is bounded above for all $t\ge 0$. By Proposition \ref{col:comp1}, $\Vert K\circ \Phi^t(\xi)\Vert$ is also bounded, so we know $h(t)$ is bounded above for $t \geq 0$.
%Next, observe that $\Vert\Phi^{s+t}\xi\Vert\le h(t)\Vert\Phi^{s}\xi\Vert\le h(t)h(s)$. Since this holds for all $\xi$, we have $h(s+t)\le h(s)h(t)$. Using Proposition \ref{prop:crt5}, we have that the limit as $t$ tends to infinity of $h(t)$ is 0. Therefore, we can find $T>0$ such that for all $t\ge T$, $h(t)<\epsilon<1$. By iterating, we have $h(nT)\le h(T)^n<\epsilon^n$. For all $t>T$, we let $n$ be such that $nT\le t<(n+1)T$. Observe that this implies $h(t)\le \left(h (t/n)\right)^n< \epsilon^n=e^{n\ln\epsilon}\le e^{(t / T-1)\ln \epsilon}$. Thus, letting $d\coloneqq \max(\epsilon^{-1},\epsilon^{-1} \sup\{ h(t) \ | \ t \in [0,T]\})$ and $c\coloneqq-\ln(\epsilon)/T$, we have $h(t)<de^{-ct}$ for all $t\ge 0$.
\end{proof}

\section{Proof of Theorem \ref{thm:main} and Corollary \ref{cor:main2}} \label{sec:main}

We now have the tools to prove Theorem \ref{thm:main}.

\begin{proof}[Proof of Theorem \ref{thm:main}] The result follows by showing \eqref{thm:1} $\Leftrightarrow$ \eqref{thm:3} and \eqref{thm:2} $\Leftrightarrow$ \eqref{thm:3}.

First, we show \eqref{thm:1} $\Rightarrow$ \eqref{thm:3}. Suppose there is a bounded perpendicular Jacobi field $J^\perp(t)$ whose upper bound is given by ${\sup_{t\in\mathbb R}\vert J^\perp(t)\vert\le C}$ for some $C>0$. 
\iffalse
Proposition \ref{prop:crt1} implies that this $J^\perp_{(x,v)}(t)$ is associated to a vector ${\xi\in Q^b(x,v)}$ that lives in both ${\hat E^{-,b}(x,v)}$ and $\hat E^{+,b}(x,v)$.  
It follows from Proposition \ref{Invariance} that $\Phi^t(\xi)\in \hat E^{-,b}(\gamma_{(x,v)}(t),\dot{\gamma}_{(x,v)}(t))\bigcap\hat E^{+,b}(\gamma_{(x,v)}(t),\dot{\gamma}_{(x,v)}(t))$. 
This says that $\Vert d\pi\circ\Phi^t(\xi)\Vert\le C$ for all $t\in\mathbb R$.
\fi
The perpendicular Jacobi field $J^\perp(t)$ is associated to a vector ${\xi\in Q^b(v)}$ satisfying $\Vert d\pi\circ\Phi^t(\xi)\Vert\le C$ for all $t\in\mathbb R$. Proposition \ref{prop:crt1} implies that ${\xi\in Q^b(v)}$ lives in both ${\hat E^{-}(v)}$ and $\hat E^{+}(v)$.
%, and it follows from Proposition \ref{Invariance} that $\Phi^t(\xi)\in \hat E^{-,b}(\gamma_{(x,v)}(t),\dot{\gamma}_{(x,v)}(t))\bigcap\hat E^{+,b}(\gamma_{(x,v)}(t),\dot{\gamma}_{(x,v)}(t))$. 
%
Corollary \ref{col:comp1} then gives us $\Vert K\circ\Phi^t(\xi)\Vert \le k \Vert d\pi\circ\Phi^t(\xi)\Vert,$
and so by definition we have $\Vert \Phi^t(\xi)\Vert\le (1+k^2)^{\frac{1}{2}}C.$

By \eqref{thm:1}, we know that we can write $\xi=\xi^++\xi^-$, with $\xi^+\in \hat E^+$ and $\xi^-\in \hat E^-$. If they are both nonzero, then by letting $t$ tend to $+\infty$, we would have $\Vert\Phi^t(\xi^-)\Vert$ tend to $+\infty$ and $\Vert\Phi^t(\xi^+)\Vert$ tend to $0$. Since this is a bounded perpendicular Jacobi field, $\xi^-$ must be $0$. A similar argument with $t$ tending to $-\infty$ yields that $\xi^+=0$. Thus, the bounded perpendicular Jacobi field must be trivial.

Next, we show \eqref{thm:3} $\Rightarrow$ \eqref{thm:1}. If there does not exist a non-trivial bounded perpendicular Jacobi field, it follows from Proposition \ref{prop:crt6} that there are two numbers $c,d>0$ so that for all $v \in SM$ and $\xi\in \hat E^{\pm}_{v}$, we have $\Vert \Phi^{\pm t}\xi\Vert\le de^{-ct}\Vert\xi\Vert.$ This and Proposition \ref{prop:Invariance} implies that the Green bundles $\hat{E}^{\pm}$ give us the necessary splitting for $\Phi^t$ to be Anosov.

We now show \eqref{thm:2} $\Rightarrow$ \eqref{thm:3}. Suppose there is a bounded perpendicular Jacobi field $J^\perp(t)$. Proposition \ref{prop:crt1} implies that it lives in both $\hat E^{+}(v)$ and $\hat E^{-}(v)$, so if $\hat E^{+}{(v)}\cap \hat E^{-}{(v)}=\{0\}$, then we immediately have $J^\perp(t)=0$.

Finally, we show \eqref{thm:3} $\Rightarrow$ \eqref{thm:2}. By Proposition \ref{prop:crt4}, we know the perpendicular Jacobi field is bounded if and only if $\xi\in \hat E^{+}(v)\cap \hat E^{-}(v)$, so $\hat E^{+}(v)\cap \hat E^{-}(v)=\{0\}$. \qedhere
\end{proof}

With Theorem \ref{thm:main} in hand, we now prove the following intermediate corollary, which will be used to prove Corollary \ref{cor:main2}.

\begin{corollary} \label{cor:main3}
Let $M$ be a closed, oriented surface, and let $(g,b)$ be a magnetic system such that for any unit speed magnetic geodesic $\gamma$, we have $\frac{d}{dt}\vert J^{\perp,z}(t)\vert \ge0$ for all $t>0$. The magnetic flow $\smash{\phi_{g,b}^t} : SM \rightarrow SM$ is Anosov if and only if for every $v \in SM$, \hbox{$\{t \in \R \ | \ K^{g,b}(\gamma_v(t), \dot{\gamma}_v(t)) < 0\} \neq \varnothing$}.
\end{corollary}

Note that $\K^{g,b} \leq 0$ and \eqref{eqn:jacobi_defn} implies that $\frac{d^2}{dt^2} |J^{\perp,z}(t)| \geq 0$, hence $\frac{d}{dt} |J^{\perp,z}(t)| \geq 0$ (see also \cite[Proposition 3.2]{IVJ}). Thus, Corollary \ref{cor:main2} follows from Corollary \ref{cor:main3}.

\begin{proof}[Proof of Corollary \ref{cor:main3}]
    First, notice that the assumption implies that $(g,b)$ is without conjugate points (see  \cite[Proposition 3.2]{IVJ}).
    Next, given the condition ${\vert J^{\perp,z}(t)\vert \ge \vert J^{\perp,z}(s)\vert}$ for $t\ge s>0$, 
    we can use the conclusion of Proposition \ref{prop:crt4} as mentioned in Remark \ref{rem:crt3.5}. Thus, we know that $\xi\in \hat E^{\pm}(v)$ if and only if $\Vert d\pi\circ \Phi^{\pm t}(\xi)\Vert$ is bounded above $\text{for all }t\ge 0$. The goal is now to use Corollary \ref{cor:main} to get the desired result. First, we start by proving the implication via the contrapositive. Assume that there is a magnetic $\gamma$ along which we have $\K^{g,b}(\gamma(t), \dot{\gamma}(t)) \geq 0$ for all $t$. In particular, this means that $\K^{g,b}(\gamma(t), \dot{\gamma}(t)) \geq -k^2$ for all $t \in \R$ and $k > 0$. For a globally defined solution to \eqref{eqn:R} along $\gamma$, we have that $u \equiv 0$ by  Proposition \ref{prop:comp1} \eqref{item:prop_comp1_2}. This implies that we have $\mathbb K^{g,b}(\gamma(t), \dot{\gamma}(t)) = 0$ for all $t \in \R$, and hence $J^\perp(t) = 1$ is a solution to the perpendicular Jacobi equation in \eqref{eqn:jacobi_defn}. By Corollary \ref{cor:main}, this implies that $\smash{\phi_{g,b}^t}$ is not Anosov. 

    Next, we show that if $J^\perp(t) = 1$ is not a perpendicular Jacobi field for any magnetic geodesic, then $\smash{\phi_{g,b}^t}$ is Anosov. In particular, we will use the assumption of the corollary to show that if $J^\perp(t)$ is bounded, then it must be constant. Suppose $J^\perp(t)$ is bounded, and suppose it corresponds to $\xi \in Q^b(v)$ under the correspondence in Lemma \ref{lem:correspondence}. Using Proposition \ref{prop:crt4}, we have that $\xi \in \hat{E}^+(v) \cap \hat{E}^-(v)$. Moreover, by the proof of Proposition \ref{prop:crt4}, we see that $\max\{|J^\perp(t)|, \vert J^\perp(-t)| \} \leq |J^\perp(0)|$ for all $t \geq 0$. In particular, this implies that $|J^\perp(t)|$ is non-increasing for $t \geq 0$.

    Now fix $t \geq 0$ and consider the function $J^\perp_t(s) \coloneqq J^\perp(t+s)$. This is still a perpendicular Jacobi field which is bounded, so by the argument from the previous paragraph we have $|J^\perp_t(-s)| \leq |J^\perp_t(0)|.$ If we take $s = t$, then this implies that $|J^\perp(0)| \leq |J^\perp(t)|$. Since this holds for all $t \geq 0$, we see that $|J^\perp(t)|$ is constant, and using the correspondence from Lemma \ref{lem:perpjacobflip} we have that it is constant everywhere.

    With this in hand, we prove the converse via the contrapositive. More precisely, we wish to show that if $\smash{\phi_{g,b}^t}$ is not Anosov, then there is a magnetic geodesic $\gamma$ along which ${\K^{g,b}(\gamma(t), \dot{\gamma}(t)) = 0}$. Since the flow is not Anosov, the previous claim implies that $J^\perp(t)$ is a perpendicular Jacobi field for some magnetic geodesic $\gamma$. As described in Section \ref{sec:comparisons}, we have that $u(t) \coloneqq (J^\perp)^{\cdot}(t)/J^\perp(t)$ is a solution to \eqref{eqn:R} along $\gamma$. Since $J^\perp$ is constant, this implies that $u \equiv 0$ is a solution to \eqref{eqn:R} along $\gamma$, which gives the desired result.     \qedhere
\end{proof}

\bibliographystyle{abbrv}
\bibliography{bibliography.bib}
\end{document}